\documentclass[12pt]{amsproc}
  \usepackage{latexsym} 
  \usepackage[all]{xy}
  \usepackage{amsfonts} 
  \usepackage{amsthm} 
  \usepackage{amsmath} 
  \usepackage{amssymb}
  \usepackage{pifont}  
  \usepackage{enumerate}
  \xyoption{2cell}

\newcommand{\cC}{{C}}

  \def\sw#1{{\sb{(#1)}}}

  \def\suc#1{{\sp{(#1)}}}

  \def\<{{\langle}} 
  \def\>{{\rangle}} 
  \def\ra{{\triangleleft}} 
   
  \def\eps{\varepsilon}

  \def\note#1{{}} 
 \def\can{{\rm \textsf{can}}}

  \def\note#1{} 

  \def\cM{{\mathcal M}} 
   
  \def\cN{{\mathcal N}} 
   \def\cQ{{\mathcal Q}} 
  \def\cC{{\mathcal C}} 
  \def\cA{{\mathcal A}} 
  \def\cB{{\mathcal B}} 
  \def\cD{{\mathcal D}} 
  \def\cE{{\mathcal E}}

  \def\lrhom#1#2#3#4{{{\rm Hom}\sb{#1- #2}(#3,#4)}}

  \def\ttheta{{\tilde{\theta}}}

  \def\beq{\begin{equation}} 
  \def\eeq{\end{equation}}

  \def\id{\mathrm{id}}

  \def\ot{{\otimes}}

 \def\hchi{\hat{\chi}}

     \def \hPhi{\widehat{\Phi}}
      \def \hphi{\widehat{\phi}}
       \def \hchi{\widehat{\chi}}

 \def\ust{\underline{*}}
 \def\ubu{\underline{\bullet}}

  \newcounter{zlist}

  \newcounter{blist}

  \newcounter{rlist}

\def\stac#1{\raise-.2cm\hbox{$\stackrel{\displaystyle\otimes}{\scriptscriptstyle{#1}}$}}

\def\cten#1{\raise-.2cm\hbox{$\stackrel{\displaystyle\widehat{\otimes}}
{\scriptscriptstyle{#1}}$}}

  \def\Label#1{\label{#1}\ifmmode\llap{[#1] }\else 
  \marginpar{\smash{\hbox{\tiny [#1]}}}\fi} 
  \def\Label{\label}

  \newtheorem{proposition}{Proposition}[section]
  \newtheorem{lemma}[proposition]{Lemma} 
   
  \newtheorem{theorem}[proposition]{Theorem} 

  \theoremstyle{definition} 
  \newtheorem{definition}[proposition]{Definition}

  \theoremstyle{remark}

  \newcounter{c} 
   
  \newcommand{\etyk}[1]{\vspace{-7.4mm}$$\begin{equation}\Label{#1} 
  \addtocounter{c}{1}} 
  \renewcommand{\]}{\ifnum \value{c}=1 $$\else \end{equation}\fi} 
\newcommand{\Mor}{{\sf Mor}}

\newcommand{\Bicom}{{\sf Bicom}}

\def\ot{\otimes}

\def\id{\textrm{{\small 1}\normalsize\!\!1}}

\newcommand{\Cc}{\mathcal{C}}

\def\B{{\bf B}}

\def\M{{\bf M}}

\def\*C{{}^*\hspace*{-1pt}{\Cc}}

\def\text#1{{\rm {\rm #1}}}

\def\Set{\mathbf{Set}}

\def\vect{\mathbf{Vect}}
\def\mod{\mathbf{Mod}}
\def\comod{\mathbf{Comod}}
 \def\Int#1{{\mathbf{Cat}(#1)}} 
 \def\l{\mathbf{l}}
 \def\r{\mathbf{r}}

 \def\t{\mathbf{t}}
 \def\v{\mathbf{r}}
 \def\w{\mathbf{l}}
 \def\y{\mathbf{y}}
 \def\f{\mathbf{f}}
 \def\g{\mathbf{g}}
 \def\h{\mathbf{h}}
 \def\k{\mathbf{k}}
 \def\coten#1{\Box_{#1}} 
\def\Alg#1{\mathrm{Alg}_{#1}} 
\def\Bicom#1{\mathbf{Comod}{(#1)}} 
 \def\Intc#1{{\mathbf{Cat}^{\mathsf c}(#1)}}
 \def\1{\mathbf{1}}
 \def\KL#1{\mathbf{KL}{(#1)}}

\begin{document} 

\title{Internal Kleisli categories} 
 \author{Tomasz Brzezi\'nski}
 \address{ Department of Mathematics, Swansea University, 
  Singleton Park, \newline\indent  Swansea SA2 8PP, U.K.} 
  \email{T.Brzezinski@swansea.ac.uk}   
\author{Adrian Vazquez Marquez}
 \address{ Department of Mathematics, Swansea University, 
  Singleton Park, \newline\indent  Swansea SA2 8PP, U.K.} 
\email{397586@swansea.ac.uk}

    \date{October 2009} 
%  \subjclass[2000]{58B34} 
  \begin{abstract} 
  A construction of Kleisli objects in 2-categories of noncartesian internal categories or categories internal to monoidal categories is presented.
  \end{abstract} 
  \maketitle 
  
\section{Introduction}
Internal categories within monoidal categories have been introduced  and studied by M.\ Aguiar in his PhD thesis \cite{Agu:int} as a framework for analysing properties of quantum groups. By choosing the monoidal category $\M$ appropriately, algebraic structures of recent interest in Hopf algebra (or quantum group) theory, such as corings and $C$-rings, can be interpreted as internal categories. Internal categories can be organised into two different 2-categories. The first one, denoted $\Int\M$, has internal functors as 1-cells, and internal natural transformations as 2-cells. The second one, denoted $\Intc\M$, has internal cofunctors as 1-cells, and internal natural cotransformations as 2-cells. While a functor  can be interpreted as a push-forward of morphisms, a cofunctor can be interpreted as a lifting of morphisms from one internal category to the other. The aim of this paper is to study adjunctions and monads in both $\Int\M$ and $\Intc\M$ and to show that every such monad has a Klesili object.  

A monad (comonad) in the 2-category $\Int\M$, i.e., an internal functor with two natural transformations which satisfy the usual associativity and unitality conditions, is simply called an {\em internal monad} (respectively, an {\em internal comonad}). We show that every internal monad (comonad) arises from and gives rise to a pair of adjoint functors, by explicitly constructing the Kleisli internal category. A monad  in the 2-category $\Intc\M$, i.e., an internal cofunctor with two natural cotransformations which satisfy the usual associativity and unitality conditions, is  called an {\em internal opmonad}. Similarly as for internal monads, we show that every internal opmonad  arises from and gives rise to a pair of adjoint cofunctors, by explicitly constructing Kleisli objects in $\Intc\M$. 

The paper is organised as follows. In Section~\ref{sec.cat} we review internal categories and describe 2-categories $\Int\M$, $\Intc\M$, following \cite{Agu:int}. In Section~\ref{sec.ftm} we first review elements of the formal theory of monads of Street and Lack \cite{Str:for}, \cite{LacStr:for}, which is a formal framework for studying monads in any bicategory. We then prove that $\Int\M$, $\Intc\M$ and the vertical dual of $\Int\M$,  embed into the bicategory of Kleisli objects associated to particular bicategories. The existence of these embeddings is then used in Section~\ref{sec.Kle} to deduce the existence of Kleisli internal categories (for internal monads, opmonads and comonads). The paper is completed with explicit examples which interpret  internal adjunctions and Klesli objects as twisting of (co)rings, and make a connection between the internal Kleisli categories  and constructions familiar from Hopf algebra and Hopf-Galois theories.

\section{Noncartesian internal categories}\label{sec.cat}
\subsection{The definition of an internal category.}
Extending the classical approach to internal category theory, M.\ Aguiar introduced the following notion of an {\em internal category in a monoidal category} in \cite{Agu:int}. Let $\M = (\M,\ot, \1)$  be a monoidal category (with tensor product $\ot$ and unit object $\1$), {\em regular} in the sense that it has equalisers and $\ot$ preserves all equalisers. Typical examples we are interested in are the category $\vect_k$ of vector spaces over a field $k$ and the  \underline{opposite} of the category $\mod_k$ of modules over a commutative, associative, unital ring $k$, with the standard tensor products. As a rule, when dealing with monoidal categories, we do not write canonical associativity and left and right unitality isomorphisms nor the canonical morphisms embedded in definitions of equalisers, etc.

The regularity condition ensures that a bicategory $\Bicom \M$ of comonoids can be associated to $\M = (\M,\ot, \1)$. The objects (0-cells) in $\Bicom \M$ are coassociative and counital  comonoids $\cC = (C,\Delta_C, e_C)$ ($\Delta_C$ is the comultiplication, $e_C$ denotes the counit). The 1-cells $\cC \to \cD$ are $\cC$-$\cD$-bicomodules $\cM = (M , \lambda, \varrho)$. Here $\lambda: M\to C\ot M$ denotes the left $\cC$-coaction and $\varrho: M\to M\ot D$ is the right $\cD$-coaction that are counital, coassociative and commute with each other in the standard way. Bicolinear maps are the 2-cells of  $\Bicom \M$. The vertical composition (of 2-cells) is the same as the composition in $\M$. The horizontal composition (of 1-cells and 2-cells) is given by the {\em cotensor product}: Start with comonoids $\cC$, $\cD$, $\cE$, a $\cC$-$\cD$-bicomodule $\cM = (M , \lambda_M, \varrho_M)$ and a $\cD$-$\cE$-bicomodule $\cN = (N , \lambda_N, \varrho_N)$. Define $M\coten D N$ as the equaliser:
$$\xymatrix{ M\coten D N \ar[r]& M\ot N
\ar@<0.5ex>[rr]^-{\varrho_M\ot N}\ar@<-0.5ex>[rr]_-{M \ot \lambda_N} & &
M\ot \cD \ot  N }.$$
The $\cC$-$\cE$-bicomodule $\cM\coten \cD \cN = (M\coten D N, \lambda_M\coten D N, M\coten D\varrho_N)$ is then the composite 1-cell 
$$
\xymatrix{
\cC \ar[r]^\cM & \cD \ar[r]^\cN & \cE} .
$$

Given a map of comonoids $f: \cC\to \cD$, any left $\cC$ comodule $\cA = (A, \lambda)$ can be made into a left $\cD$-comodule with the induced coaction $(f\ot A)\circ \lambda$. We denote the object part of the resulting $\cD$-comodule by ${}^f\! A$ (obviously, ${}^f\!A$ is isomorphic to  $A$ in $\M$), and the comodule itself by ${}^f\! \cA$. Similarly, any right $\cC$-comodule $\cA = (A,\varrho)$ is a right $\cD$-comodule with coaction $(A\ot f)\circ \varrho$; we denote it by $\cA^f$. 

Let  $\cC = (C,\Delta_C, e_C)$ be a comonoid. Then the category of endo-1-cells on $\cC$ (in $\Bicom \M$) or the category of $\cC$-bicomodules ${}^\cC \M^\cC$ is a monoidal category with the monoidal product given by the  cotensor product $\coten C$ and with the unit object $\cC$ (understood as a $\cC$-bicomodule with both coactions provided by $\Delta_C$).  Following  \cite{Agu:int}, by an {\em internal category in $\M$} (with object of objects $\cC$) we mean a monoid in ${}^\cC \M^\cC$, that  is a $\cC$-bicomodule with an associative and unital composition. This means, an internal category  is  a pair: a comonoid $\cC =(C, \Delta_C, e_C)$ in $\M$ and a  monoid $\cA = (A,m_A,u_A)$ in a monoidal category of $\cC$-bicomodules ${}^\cC \M^\cC$ (with cotensor product). Here $(A,\lambda,\varrho)$ is the underlying object, the multiplication in $\cA$ is denoted by $m_A$ and the unit by $u_A$. $\cA$ is thought of as the {\em object of morphisms} (with composition $m_A$ and identity morphism $u_A$), while $\cC$ is understood as an object of objects. For an internal category we write $(\cC,\cA)$ to indicate both: objects and morphisms. The left coaction $\lambda: A\to C\ot A$ should be understood as determining the {\em codomain} and the right coaction $\varrho: A\to A\ot C$ as determining  the {\em domain}.  The notation $m_A^n$ is used to record the  $n-1$-fold composite of $m_A$, e.g.\ $m^2_A = m_A\circ (m_A\coten C A)$ etc.

In case $\M=\Set$  and  $\ot = \times$ (the cartesian monoidal category) the above defines the usual internal category in $\Set$. If $\M$ is the category of vector spaces over a field, then internal category in $\M$ coincides with the notion of a {\em $C$-ring}. Of special algebraic interest is the case $\M =  \mod_k^{op}$  (the opposite of the category of modules over a commutative ring $k$). The comonoids in $\mod_k^{op}$ are the same as the monoids in $\mod_k$, i.e.\ $k$-algebras $A$; bicomodules in $\mod_k^{op}$ become $A$-bimodules. The equalisers in $\mod_k^{op}$  are the same as coequalisers in $\mod_k$; consequently the cotensor products in $\mod_k^{op}$ coincide with the usual tensor products of bimodules. Multiplication and unit become comultiplication and a counit respectively. In a word: an internal category in $\mod_k^{op}$ is the same as an {\em $A$-coring}; see \cite{BrzWis:cor} for more details about corings. We return to  examples of this in Section~\ref{sec.ex}.

While the definition of a non-cartesian internal category follows quite naturally the cartesian case, the definition of internal functors and natural transformations leaves more scope for freedom. Two possible definitions are proposed in \cite{Agu:int} and we describe them both presently. 
 
 \subsection{Internal functors and the 2-category $\Int \M$}\label{sec.cat(m)}

An {\em internal functor}  or simply a {\em functor} $\f : (\cC,\cA)\to (\cD,\cB)$ is a pair $\f = (f_0,f_1)$, where $f_0: \cC \to \cD$ is a morphism of comonoids, and $f_1: {}^{f_0}\!A ^{f_0}\to B$ is a $\cD$-bicomodule map that is multiplicative and unital in the sense that the following diagrams commute:
$$
\xymatrix{
A\square_{C}A\ar[r]\ar[d]_{m_{A}}  & A^{f_{0}} \square_{D} {}^{f_{0}}\! A\ar[r]^-{f_{1}\square_{D}f_{1}} & B\square_{D}B\ar[d]^{m_{B}} \\
A \ar[rr]_{f_{1}}& & B,
}  \qquad 
\xymatrix{
C\ar[r]^{f_{0}}\ar[d]_{u_{A}} & D\ar[d]^{u_{B}} \\
A\ar[r]_{f_{1}} & B .
} 
$$
Internal functors are composed component-wise. 
In order to relieve the notation, in what follows we  write $V^\f$ for $V^{f_0}$ etc. 
Typically, we also denote the composition of morphisms in $\M$ by juxtaposition. The composition in $\M$ takes precedence over all other operations on morphisms in $\M$.
 
An {\em internal natural transformation} $
\alpha: \f \Rightarrow\g
$ is a morphism of $\cD$-bicomodules (i.e.\ a 2-cell in $\Bicom \M$)
$\alpha: {}^{\g}C^{\f}\to B$ such that
 the following diagram commutes
\begin{equation*}
\xymatrix{
 											& A\square_{C} C\ar[r]         & A^{\g}\square_{D}{}^{\g} C\ar[r]^{g_{1}\square_{D}\alpha}& B\square_{D}B\ar[rd]^{m_{B}}  &     & \\
A\ar[ru]^{\varrho}\ar[rd]_{\lambda}	&                                       & 																				 &  						&  B \\
        									& C\square_{C}A\ar[r]          & C^{\f}\square_{C}{}^{\f}\! A \ar[r]_{\alpha\square_{D}f_{1}}  & B\square_{D}B\ar[ru]_{m_{B}}. &    
}
\end{equation*}

Natural transformations can be composed in two different ways. The {\em vertical composition} is given by the {\em convolution product $*$}. That is, the composite of 
$\alpha:\f \Rightarrow \g$ and $\beta: \g \Rightarrow \h$ is 
\begin{equation*}
\beta*\alpha := m_{B} \circ (\beta \square_{D} \alpha) \circ \Delta_{C}.
\end{equation*} 
The {\em horizontal composition $\bullet$}, also known as the {\em Godement product}, is defined as follows. In the situation:
$$
\xy *+{(\cC,\cA)}="A",
+<3cm,0pt>*+{(\cD,\cB)}="B", 
+<3cm,0pt>*+{(\cD',\cB')}="C",
"A";"B",{\twocell^\f_\g{\alpha}},
"B";"C",{\twocell^\h_\k{\beta}}
\endxy
$$
 the Godement product $\beta\bullet\alpha: \h\f \Rightarrow \k\g$ is defined as
\begin{equation*}
\beta\bullet \alpha : = \beta g_{0}* h_{1}\alpha = k_{1}\alpha * \beta f_{0}
\end{equation*}

With any internal functor  $\f : (\cC,\cA)\to (\cD,\cB)$ one associates the {\em identity natural transformation on $\f$}. This is given as 
$$
f:= f_1\circ u_A = u_B \circ f_0 : C\to B.
$$
The second equality follows by the unitality of functors. In view of the unitality of multiplications, $\alpha * f =\alpha$ and $f*\beta = \beta$, for any natural transformations $\alpha$, $\beta $ with domain and codomain $\f$, respectively. 

A collection of internal categories in $\M = (\M,\ot, \1)$ together with internal functors (with composition) and internal natural transformations (with vertical and horizontal compositions, and the identity natural trasformations as units for the vertical composition) forms a 2-category $\Int\M$. 
In case $\M = \mod_k^{op}$ 0-cells of $\Int{\mod_k^{op}}$ are corings, 1-cells are coring morphisms and the opposite of 2-cells are representations of corings; see \cite[Section~24]{BrzWis:cor}.

 \subsection{Internal cofunctors and the 2-category $\Intc \M$}
 An {\em internal cofunctor} or simply a {\em cofunctor} $\f : (\cC,\cA)\to (\cD,\cB)$ is a pair $\f = (f_0,f_1)$, where $f_0: \cD \to \cC$ is a morphism of comonoids, and $f_1: A\coten C {}^{f_0}\!D \to  {}^{f_0}\! B$ is a $\cC$-$\cD$-bicomodule map that respects multiplications and units in the sense that the following diagrams commute:
 $$
 \xymatrix{
A\square_{C}A\coten C {}^{\f}\!D \ar[r]^-{A\coten C f_1}\ar[d]_{m_{A}\coten C D}  & A \square_{C} {}^{\f}\! B\ar[r]^-{\simeq} &A \square_{C} {}^{\f}\! D\coten D B  \ar[d]^{f_1\coten D B}\\
A \coten C {}^{\f}\! D \ar[r]_-{f_{1}}& B & B\coten D B\ar[l]^{m_B} ,
}  \qquad 
\xymatrix{
A\coten C {}^{\f}\! D\ar[r]^-{f_{1}} & B   \\
 & D\ar[lu]^{u_{A}\coten C  D} \ar[u]_{u_{B}} ,
} 
$$
where we write ${}^\f \!B$ for ${}^{f_0}\!B$ etc., as in Section~\ref{sec.cat(m)}. The composite $\h\ubu \f$ of two cofunctors $\f : (\cC,\cA)\to (\cD,\cB)$, $\h : (\cD,\cB)\to (\cD',\cB')$ has the object component $(\h\ubu \f)_0 = f_0\circ h_0$, and its morphism component is the composite (in $\M$):
$$
\xymatrix{ A\coten C {}^{f_0\circ h_0}\!D' \ar[r]^-{\simeq} & A\coten C {}^\f\! D\coten D {}^\h\!D' \ar[rr]^-{f_1\coten D D'} && B\coten D {}^\h\!D' \ar[r]^-{h_1} & B' .}
$$

Let $\f,\g: (\cC,\cA)\to (\cD,\cB)$ be cofunctors. A {\em natural cotransformation} $\alpha: \f \Rightarrow\g
$ is a morphism of $\cD$-bicomodules (i.e.\ a 2-cell in $\Bicom \M$)
$\alpha: {}^{\f}D\to {}^{\g}\!B$ such that
 the following diagram commutes
$$
\xymatrix{& A\square_{C} {}^\g\! B\ar[r]^-{A\coten C \lambda}      & A\coten {C}{}^{\g}\!D\coten D B \ar[r]^-{g_{1}\coten{C} B}& B\square_{D}B\ar[rd]^{m_{B}}  &      \\
A\coten C ^{\f}\!D \ar[ru]^{A\coten C \alpha}\ar[rd]_{f_1}&& & & B \\
& B\ar[r]_-\lambda   & D\square_{D}B \ar[r]_-{\alpha\square_{D}B}  & B\square_{D}B\ar[ru]_{m_{B}} . &    
}
$$
The vertical composition of natural cotransformations $\alpha$, $\beta$ between cofunctors from $(\cC,\cA)$ to $(\cD,\cB)$ is the composite
$$
\beta \ust \alpha : \xymatrix{ D \ar[r]^-\alpha & B\ar[r]^-\lambda & D\coten D B\ar[r]^-{\beta\coten D B} & B\coten D B \ar[r]^-{m_B} & B}.
$$
The identity natural cotransformation for a cofunctor  from $(\cC,\cA)$ to $(\cD,\cB)$ is equal to $u_B$. In the situation:
$$
 \xy  *+{(\cC,\cA)}="A",
+<3cm,0pt>*+{(\cD,\cB)}="B", 
+<3cm,0pt>*+{(\cD',\cB')}="C",
"A";"B",{\twocell^\f_\g{\alpha}},
"B";"C",{\twocell^\h_\k{\beta}} ,
\endxy 
$$
the horizontal composition $\beta\ubu\alpha : \h\ubu\f \Rightarrow \k\ubu \g$ is the following composite:
$$
\xymatrix{ {}^{\h\ubu\f}\! D' \! \ar[r]^-\simeq & \!{}^\f\! D\coten D{}^\h\! D' \ar[r]^-{\alpha\coten D \beta} & {}^\g\! B\coten D\! {}^\k\! B' \! \ar[r]^-\simeq & \! {}^\g\! B\coten D{}^\k\! D'\coten{D'} B' \ar[r]^-{k_1\coten {D'} B'} & {}^{\k\ubu\g}\! B'\coten {D'} B'\! \ar[r]^-{m_B} & \! {}^{\k\ubu \g}\! B'.}
$$

As explained in \cite[Section~4.3]{Agu:int}, while an internal functor can be understood as a push-forward of arrows from one category to the other, an internal cofunctor can be interpreted as a lifting of arrows. In the case of $\M = \mod_k^{op}$, a cofunctor $\f$ from an $A$-coring $C$ to a $B$-coring $D$ is equivalent to the following commutative diagram of functors:
$$
\xymatrix{ D\! - \! \comod \ar[rrrr]^{F_1}\ar[dr] &&&& C\! - \! \comod \ar[dl] \\
& B\! - \! \mod \ar[rr]^{F_0}\ar[dr] && A\! - \! \mod \ar[dl]& \\
&& \mod_k .&&}
$$
Here $C\! - \! \comod$, $A\! - \! \mod$ etc.\ denote categories of left comodules respectively modules. The unmarked arrows are the forgetful functors and $F_0$ is the restriction of scalars functor corresponding to the $k$-algebra map $f_0: A\to B$. Thus a cofunctor in this case can be interpreted as a left extension of corings \cite{Brz:not}. 

 \section{Internal categories and the formal theory of monads}\label{sec.ftm}
 Following \cite{KelStr:rev}, a {\em Kleisli object}  of a monad $\t$ in a 2-category is an object representing the covariant $\t$-algebra functor  to the 2-category of categories  (the reader not familiar with this concept should consult Section~5 of \cite{MacSob:asp}). A formal approach to the theory of monads was subsequently developed in \cite{Str:for} and \cite{LacStr:for}. Although not every monad in a given bicategory has a Kleisli object, to any bicategory, say $\B$,  a new bicategory $\KL \B$ can be associated. Any monad in $\KL \B$ has a Klesli object (and it is given in a form of a {\em wreath product}). The detailed description of $\KL\B$ is given in \cite[Section~1]{LacStr:for}; here we give the relevant formulae in the case $\B = \Bicom \M$.
 
 The 0-cells of $\KL{ \Bicom \M}$ are monads in $ \Bicom \M$, i.e.\ monoids in the categories of $\cC$-bicomodules, i.e.\ internal categories $(\cC,\cA)$ in $\M$. A 1-cell $(\cC,\cA) \to (\cD,\cB)$ in $\KL{ \Bicom \M}$ is a pair $(\cM,\phi)$, where $\cM = (M,\lambda_M,\varrho_M)$ is a $\cC$-$\cD$-bicomodule and 
 $$
 \phi: A\coten C M \to M\coten D B
 $$
 is a $\cC$-$\cD$-bicomodule map rendering commutative the following diagrams
 $$
 \xymatrix{
A\square_{C}A\coten C M \ar[r]^-{A\coten C \phi}\ar[d]_{m_{A}\coten C M}  & A \square_{C} M \coten D B  \ar[r]^-{\phi\coten D B} & M\coten D B\coten D B\ar[d]^{M\coten D m_B} \\
A \coten C M \ar[rr]_-{\phi}  && M\coten D B ,
}  \qquad 
\xymatrix{
A\coten C M\ar[r]^-{\phi} & M\coten D B   \\
 & M\ar[lu]^{u_{A}\coten C  M} \ar[u]_{M\coten D u_{B}} .
} 
$$
A 2-cell
$$
\xy *+{(\cC,\cA)}="A",
+<3cm,0pt>*+{(\cD,\cB)}="B", 
"A";"B",{\twocell^{(\cM,\phi)}_{(\cN,\psi)}{\chi}},
\endxy 
$$
 is a $\cC$-$\cD$-bicomodule map $\chi: M\to N\coten D B$ rendering commutative the following diagram
 $$
 \xymatrix{
A\square_{C}M \ar[r]^-{\phi}\ar[d]_{A\coten C \chi}  & M \square_{D} B\ar[r]^-{\chi\coten D B} &N\coten D B\coten D B   \ar[d]^{N\coten D m_B}\\
A \coten C N\coten D B \ar[r]_-{\psi\coten D B}& N\coten D B \coten D B\ar[r]_-{N\coten D m_B}& N\coten D B .
} 
$$
 The vertical composition of 2-cells $\chi: (\cM,\phi) \Rightarrow (\cN,\psi)$, $\chi':(\cN,\psi) \Rightarrow (\cQ,\theta)$ is the composite morphism
 $$
 \xymatrix{ M \ar[r]^-\chi & N\coten D B \ar[r]^-{\chi'\coten D B} & Q\coten D B \coten D B \ar[r]^-{Q\coten D m_B} & Q\coten D B.}
 $$
The horizontal composition of 
 $$
\xy *+{(\cC,\cA)}="A",
+<3cm,0pt>*+{(\cD,\cB)}="B", 
+<3cm,0pt>*+{(\cD',\cB')}="C",
"A";"B",{\twocell^{(\cM,\phi)}_{(\cN,\psi)}{\chi}},
"B";"C",{\twocell^{(\cM',\phi')}_{(\cN',\psi')}{\chi'}} ,
\endxy 
$$
 is the composite
 $$
 \xymatrix{ M\coten D M'\ar[r]^-{\chi\coten D M'} & N\coten DB\coten DM' \ar[r]^-{N\coten D \phi'} & N\coten D M'\coten {D'} B'  && \\
&  \ar[r]^-{N\coten D \chi'\coten {D'} B'} & N\coten D N'\coten {D'} B'\coten{D'} B'\ar[rr]^-{N\coten D N'\coten {D'}m_{B'}} && N\coten D N'\coten {D'} B'.}
 $$
 As all $\Int\M$, $\Intc\M$ and $\KL{ \Bicom \M}$ have the same 0-cells it is natural to ask, what is the relationship of the former two 2-categories to the latter bicategory. This question is addressed presently. For any bicategory $\B$, $\B^*$ denotes the bicategory obtained by reversing 1-cells in $\B$, while $\B_*$ is the bicategory obtained by reversing 2-cells in $\B$. 

 \begin{theorem}\label{thm.meta1}
 $\Int\M$ is locally fully embedded in $\KL{\Bicom \M}$ and $\Int\M_*$ is locally fully embedded in $\KL{\Bicom \M^*}$.
 \end{theorem}
 \begin{proof}
 The embedding  functor 
 $$
 \Phi: \Int\M \to \KL{\Bicom \M},
 $$
 is defined as follows. On 0-cells $\Phi$ is the identity. On 1-cells
 $$
 \Phi: \left(\xymatrix{(\cC,\cA)\ar[r]^\f & (\cD,\cB)}\right)\longmapsto \left(\cC^\f, \phi_\f: A\coten C C^\f \cong A^\f \to C^\f\coten DB\right),
 $$
 where $\phi_\f = (C\coten C f_1)\circ \lambda_A$. On 2-cells
 $$
  \Phi: \left( \quad \xy  *+{(\cC,\cA)}="A",
+<3cm,0pt>*+{(\cD,\cB)}="B", 
"A";"B",{\twocell^\f_\g{\alpha}}
\endxy \quad  \right)
\longmapsto \left( \quad  \xy *+{(\cC,\cA)}="A",
+<3cm,0pt>*+{(\cD,\cB)}="B", 
"A";"B",{\twocell^{(\cC^\f,\phi_\f)}_{(\cC^\g,\phi_\g)}{~~\chi_\alpha}}
\endxy \quad  \right) ,
$$
where $\chi_\alpha = (C^\g\coten D \alpha)\circ \Delta_C$. The proof that $\Phi$ is a functor is a standard exercise in the diagram chasing. Since $\Phi$ is the identity on 0-cells it is an embedding. To see that it is locally full, take internal functors $\f, \g : (\cC,\cA) \to (\cD,\cB)$ and any 2-cell
$$
\xy *+{(\cC,\cA)}="A",
+<3cm,0pt>*+{(\cD,\cB)}="B", 
"A";"B",{\twocell^{(\cC^\f,\phi_\f)}_{(\cC^\g,\phi_\g)}{\chi}},
\endxy
$$
in $\KL{\Bicom \M}$. Then $\alpha = m_B\circ (g_{0}\coten D B)\circ \chi$ is an internal natural transformation $\alpha: \f \Rightarrow \g$. Furthermore, $\Phi(\alpha) = \chi$. Therefore, $\Phi$ is a locally full embedding as stated. 

The embedding 
$
 \hPhi: \Int\M_* \to \KL{\Bicom \M^*},
 $
 is obtained from $\Phi$ by the left-right symmetry. That is, on 0-cells $\hPhi$ is the identity. On 1-cells
 $$
 \hPhi: \left(\xymatrix{(\cC,\cA)\ar[r]^\f & (\cD,\cB)}\right)\longmapsto \left({}^\f\!\cC, \hphi_\f: {}^\f\! C \coten C A \cong {}^\f\!A \to B\coten D{}^\f\!C\right),
 $$
 where $\hphi_\f = (f_1\coten C C)\circ \varrho_A$. On 2-cells (in $\Int\M$)
 $$
  \hPhi: \left( \quad \xy  *+{(\cC,\cA)}="A",
+<3cm,0pt>*+{(\cD,\cB)}="B", 
"A";"B",{\twocell^\g_\f{\alpha}}
\endxy \quad  \right)
\longmapsto \left( \quad  \xy *+{(\cC,\cA)}="A",
+<3cm,0pt>*+{(\cD,\cB)}="B", 
"A";"B",{\twocell^{({}^\f\!\cC,\hphi_\f)}_{({}^\g \cC,\hphi_\g)}{~~\hchi_\alpha}}
\endxy \quad  \right) ,
$$
where $\hchi_\alpha = (\alpha\coten D {}^\g C)\circ \Delta_C$. 
 \end{proof}
 
\begin{theorem}\label{thm.meta2}
 $\Intc\M$ is locally fully embedded in $\KL{\Bicom \M}$.
 \end{theorem}
 \begin{proof}
 This is even more straightforward than Theorem~\ref{thm.meta1}. A functor 
 $$
 \Psi: \Intc\M \to \KL{\Bicom \M},
 $$
 is defined as follows. On 0-cells  $\Psi$ is the identity.  On 1-cells
 $$
 \Psi: \left(\xymatrix{(\cC,\cA)\ar[r]^\f & (\cD,\cB)}\right)\longmapsto \left({}^\f \cD , f_1\right),
 $$
 On  2-cells $\Psi$ is the identity, where, by  the standard identification ${}^\g B \cong {}^\g D\coten D B$,  $\alpha: {}^\f D\to {}^\g B$ is now understood as a morphism ${}^\f D\to {}^\g D\coten D B$. The aforesaid identification allows one to see immediately that $\Psi$ is well-defined and compatible with all compositions.
 \end{proof}

\section{Adjunctions and monads on noncartesian internal categories}\label{sec.Kle}

\subsection{Adjunctions and (co)Kleisli objects in $\Int \M$}\label{sec.Kleisli}

Let $(\cC,\cA)$, $(\cD,\cB)$ be internal categories in $\M = (\M,\ot, \1)$ and consider a pair of internal functors $\l : (\cC,\cA)\to  (\cD,\cB)$ and $\r : (\cD,\cB)\to (\cC,\cA)$.  The functor $\l$ is said to be {\em left adjoint} to $\r$, provided  $\l \dashv \r$ is  an adjunction in the 2-category $\Int\M$. That is  there are internal natural transformations $\eps: \l\r \Rightarrow \id$, $\eta: \id\Rightarrow \r\l$ such that
\begin{equation}\label{t.eq}
\eps l_0 * l_1\eta = l , \qquad 
r_1\eps *\eta r_0 = r. 
\end{equation}
The conditions \eqref{t.eq} are referred to as {\em triangular identities}, $\eps$ is called a {\em counit} and $\eta$ is called a {\em unit} of the adjunction $\l \dashv \r$. The aim of this section is to show that, similarly to the standard ``external"{} category theory, the adjointness can be characterised by  an isomorphism of morphism objects. 

Start with a pair of internal functors $\l : (\cC,\cA)\to  (\cD,\cB)$ and $\r : (\cD,\cB)\to (\cC,\cA)$. The object $D^\r \coten  C A$ can be interpreted as an object of all morphisms in $(\cC,\cA)$ with codomain $\r$, i.e.\ of morphisms $- \to \r (-)$. Thus, intuitively, $D^\r \coten  C A$  is an internal version of the bifunctor $\Mor_\cA ( - , \r -)$. Similarly $B\coten  D {}^\l C$ can be interpreted as a collection of arrows $\l(-) \to -$, i.e.\ as an internalisation of $\Mor_\cB ( \l- ,  -)$. With this interpretation in mind we propose the following
\begin{definition}\label{def.binat}
Let $\l : (\cC,\cA)\to  (\cD,\cB)$ and $\r : (\cD,\cB)\to (\cC,\cA)$ be a pair of internal functors.  A $\cD$-$\cC$-bicomodule map
$$
\theta: D^\r \coten  C A \to B\coten  D {}^\l C,
$$
is said to be {\em bi-natural} provided it renders commutative the following diagrams:
$$
\xymatrix{B^\r \coten C  A\!  \ar[rr]^{\varrho \coten C A} \ar[d]_{\lambda \coten C A} && \! B\coten D  D^\r \coten CA\!  \ar[rr]^{B\coten D\theta} && \! B\coten D B\coten D {}^\l C \ar[d]^{{}\! m_B\coten D \! {}^\l C} \\
{}\!  D\coten D B^\r \coten C A \ar[rd]_{D\coten D r_1\coten C A} &&&& B\coten D {}^\l C \\
& D^\r \coten C A\coten C A \ar[rr]^{D^\r \coten C m_A} && D^\r \coten C A \ar[ru]_\theta & 
}
$$
and 
$$
\xymatrix{D^\r \coten C A \coten C A \ar[r]^-{\theta \coten C \varrho} \ar[dr]_{D^\r \coten C m_A} & B\coten D  {}^\l C\coten CA\coten C C  \ar[rr]^-{B\coten Dl\coten Cl_1\coten D {}^\l C} && B\coten DB\coten D B \coten D {}^\l C \ar[dl]^{m_B^2 \coten D {}^\l C} \\
& D^\r\coten C A \ar[r]_{\theta} & B\coten D {}^\l C. &}
$$
\end{definition}

The readers can easily convince themselves that the first of the above diagrams corresponds  to the statement that  $\theta$ is a natural transformation between covariant morphism functors. Since the category of $\cC$-bicomodules is not symmetric (nor even braided), the definition of an oppostive internal category hence contravariant internal functor does not seem to be  possible. On the other hand a condition reflecting naturality of $\theta$ as a transformation between contravatiant hom-functors can be stated provided all the morphisms are evaluated (or composed) at the end. This is the essence of the second diagram.

\begin{theorem}\label{thm.adjoint}
Let $\l : (\cC,\cA)\to  (\cD,\cB)$ and $\r : (\cD,\cB)\to (\cC,\cA)$ be a pair of internal functors. Then  $\l \dashv \r$ is an adjunction if and only if the objects $D^\r \coten  C A$ and $B\coten  D {}^\l C$ are isomorphic by a bi-natural map of $\cD$-$\cC$-bicomodules.
\end{theorem}
\begin{proof} We give two proofs of this theorem. The first one is based on straightforward albeit lengthy calculations that use generalised elements. The second one, suggested to us by G.\ B\"ohm, is more conceptual and based on the interplay between $\Int\M$ and $\KL{\Bicom \M}$ (see Theorem~\ref{thm.meta1}) and description of adjunctions in the latter. \\ ~

\underline{Computational proof.} Suppose that $\l \dashv \r$ is an adjunction with counit $\eps$ and unit $\eta$. Then
the isomorphism is
$$
\theta: D^\r \coten  C A \to B\coten  D {}^\l C, \qquad \theta = (m_B\coten D {}^\l C)\circ (\eps \coten D l_1\coten C C)\circ (D^\r \coten C \varrho),
$$
with the inverse
$$
\theta^{-1} : B\coten  D {}^\l C\to D^\r \coten  C A, \qquad \theta^{-1} = (D^\r \coten C m_A)\circ (D\coten Dr_1 \coten C \eta)\circ(\lambda \coten D {}^\l C).
$$
These maps are well-defined by the colinearity of $\eps$, $\eta$ and the multiplication maps $m_A$ and $m_B$. To prove that $\theta^{-1}\circ\theta$ is an identity apart for the colinearity one uses the naturality of $\eta$, multiplicativity of $r_1$ and the second triangular identity \eqref{t.eq}. Dually, the proof that  $\theta\circ\theta^{-1}$ is an identity uses the naturality of $\eps$, multiplicativity of $l_1$ and the first triangular identity in \eqref{t.eq}.  The naturality of $\eps$ combined with colinearity of $m_A$, multiplicativity of $l_1$ and unitality of $m_A$ yield the bi-naturality of $\theta$. Dually, the bi-naturality of $\theta^{-1}$ follows by the naturality of $\eta$, colinearity of $m_B$, multiplicativity of $r_1$ and the unitality of $m_B$. 

Conversely, assume that there is a  bi-natural map $\theta: D^\r \coten  C A \to B\coten  D {}^\l C$ with the inverse $
\theta^{-1} : B\coten  D {}^\l C\to D^\r \coten  C A$. Define the bicomodule maps
$$
\eps: D^{\l\r} \to B, \qquad \eta : {}^{\r\l}C \to A,
$$
as composites
$$
\eps = m_B \circ (B\coten D l)\circ \theta \circ (D\coten D r)\circ \Delta_D, \qquad
\eta = m_A \circ (r\coten C A)\circ\theta^{-1}\circ (l \coten C C)\circ \Delta_C.
$$
The bi-naturality of $\theta$ together with the unitality of $m_A$ imply that $\eps$ is a natural transformation $\eps: \l\r \Rightarrow \id$. Similarly, $\eta$ is a natural transformation $\eta: \id \Rightarrow \r\l$ by bi-naturality of $\theta^{-1}$ and unitality of $m_B$. This is confirmed by straigthforward calculations. 

To prove the first of triangular identities \eqref{t.eq} we introduce an explicit notation for the comultiplication and left coactions, i.e.\ the Sweedler notation, which we adopt in the form $\Delta_C^n(c) = c\sw 1\ot c\sw 2\ot \cdots \ot c\sw{n+1}$ and $\lambda^n (b) = b\sw {-n}\ot \cdots \ot b\sw {-1}\ot  b\sw{0}$. Here $c$ should be understood as a generalised element of $C$ (i.e.\ a morphism from any object in $\M$ to $C$), and $b$ is a generalised element of $B$. Next, take  a generalised element  $d \coten{} a$  of $D^\r \coten C A$ and  a generalised element $b\coten {} c$  of $B\coten D{}^\l C$, and write
$$
a_\theta \coten{} d^\theta := \theta (d \coten{} a), \qquad c_\ttheta \coten{} b^\ttheta := \theta^{-1} ( b\coten{} c),
$$
for the generalised elements obtained after applying $\theta$ and $\theta^{-1}$, respectively. The properties of $\theta$ and $\theta^{-1}$ give rise to twelve equalities. There are four equalities corresponding to the colinearity (two for each $\theta$ and $\theta^{-1}$), two equalities encoding the codomains of $\theta$ and $\theta^{-1}$ (e.g.\ $a_\theta \coten{} d^\theta$  is in an apropriate equaliser), two recording that $\theta$ and $\theta^{-1}$ are mutually inverse, finally the bi-naturality of  $\theta$  results in two  equations for $\theta$ and two equations for $\theta^{-1}$. The interested reader can easily write all these equalities down; we only show them in action. 

In the forthcoming calculation, the multiplications $m_A$, $m_B$ 
as well as composition in $\M$ 
are denoted by juxtposition. We also freely use the following equalities:
\begin{equation}\label{url}
u_A r_0 l_0 = r l_0 = rl = r_1 l.
\end{equation}
etc. Take a generalised element $c$ of $C$, and, using the definition of $\eps$ and $\eta$,  and the fact that $l_0$ is a map of comonoids, compute
$$
\eps l_0 * l_1\eta (c) = \eps l_0 * l_1\eta (c) = rl_0 (c\sw 2)_\theta l\left(l_0(c\sw 1)^\theta\right)l_1 r (c\sw 4_ \ttheta) l_1 \left( l(c\sw 3)^\ttheta\right).
$$
Next, use \eqref{url} and  apply the second of the diagrams in Definition~\ref{def.binat} to the first three terms, then the first of the diagrams in Definition~\ref{def.binat} to the first three terms and then the second of the diagrams in Definition~\ref{def.binat} to the last three terms  to obtain:
\begin{eqnarray*}
\eps l_0 * l_1\eta (c) &=& \left( r_1l(c\sw 2) r(c\sw 4_\ttheta)\right)_\theta l\left(l_0(c\sw 1)^\theta\right) l_1 \left( l(c\sw 3)^\ttheta\right)\\
&=& l (c\sw 1)\left(r(c\sw 4_\ttheta)\right)_{\theta}l\left(l_0(c\sw 2)^\theta\right)l_1 \left( l(c\sw 3)^\ttheta\right)\\
&=& l (c\sw 1) \left(r(c\sw 4_\ttheta) l(c\sw 3)^\ttheta\right)_\theta l\left( l_0(c\sw 2)^\theta\right) .
\end{eqnarray*}
The left colinearity of $\theta$, unitality of multiplication $m_B$ and the left colinearity of $\theta^{-1}$ allow one to transform the above expression further:
\begin{eqnarray*}
\eps l_0 * l_1\eta (c) &=& u_B\left(\left(r(c\sw 3_\ttheta) l(c\sw 2)^\ttheta\right)_\theta\sw{-1}\right) \left(r(c\sw 3_\ttheta) l(c\sw 2)^\ttheta\right)_\theta\sw 0  l\left( l_0(c\sw 1)^\theta\right)\\
&= &\left(r(c\sw 3_\ttheta) l(c\sw 2)^\ttheta\right)_\theta  l\left( l_0(c\sw 1)^\theta\right)
= \left(r(c\sw 2_\ttheta\sw 2) l(c\sw 1)^\ttheta\right)_\theta  l\left(c\sw 2_\ttheta \sw 1^\theta\right).
\end{eqnarray*}
Finally we can use the fact that the codomain of  $\theta^{-1}$ is equal to the equaliser $D^\r\coten C A$, then the unitality of $m_A$, the property $\theta\circ \theta^{-1} = \id$, and the unitality of $m_B$ (or the idempotent property $l*l=l$) to compute
\begin{eqnarray*}
\eps l_0 * l_1\eta (c) &=&  \left(u_A\left (l(c\sw 1)^\ttheta\sw {-1}\right) l(c\sw 1)^\ttheta\sw 0\right)_\theta  l\left(c\sw 2_\ttheta{}^\theta\right)\\
&=& l(c\sw 1)^\ttheta{}_\theta  l\left(c\sw 2_\ttheta{}^\theta\right) = l(c\sw 1)l(c\sw 2) = l(c),
\end{eqnarray*}
as required. In summary, we used precisely half of the twelve properties characterising $\theta$ and $\theta^{-1}$. The other half is used to verify the second triangular equality \eqref{t.eq} (a task left to the reader). \\~

\underline{Conceptual proof.} This proof has been suggested to us by G.\ B\"ohm \cite{Boh:pri}. Start with a 2-category (or a bicategory) $\B$. 
Denote vertical composition in $\B$ by $*$ and the horizontal composition by $\bullet$. Consider a pair of 1-cells in $\KL\B$
$$
\xymatrix{(A,C) \ar@<0.5ex>[rr]^{(L,\phi)} && (B,D) \ar@<0.5ex>[ll]^{(R,\psi)} },
$$
where the notation $(A,C)$ means a monad $A$ on 0-cell $C$ (in $\B$). Assume further that $L$ has a right adjoint $K$ in $\B$, with unit $\iota$ and counit $\sigma$. As explained in \cite{Boh:pri}  there is a bijective correspondence between adjunctions $(L,\phi) \dashv (R,\psi)$ in $\KL\B$ and isomorphism 2-cells 
$$
\theta : A\!\bullet\! R \Rightarrow K\!\bullet\! B,
$$ 
satisfying the following equalities:
$$
\theta * (m_A\!\bullet\! R) = (K\!\bullet\! m_B)*(K\!\bullet\! B\!\bullet\! \sigma\!\bullet\! B)*(K\!\bullet\!\phi\!\bullet\! K\!\bullet\! B)*(\iota\!\bullet\! A\!\bullet\! K\!\bullet\! B)*(A\!\bullet\! \theta)
$$
and
$$
(K\!\bullet\! m_B)*(\theta \!\bullet\! B) = \theta *(m_A\!\bullet\! R)*(A\!\bullet\! \psi).
$$
Now take $\B = \Bicom\M$. Since there is a locally full embedding $\Phi: \Int\M \to \KL{\Bicom\M}$ (see Theorem~\ref{thm.meta1}), description of adjunctions in $\Int\M $ is the same as description of adjunctions in $\Phi(\Int\M)\subset \KL{\Bicom\M}$. Properties of the pair of internal functors $\l,\r$ can be translated to properties of 1-cells $(C^\l, \phi_\l)$ and $(D^\r, \phi_\r)$, where $\phi_\l$, $\phi_\r$ are defined in the proof of Theorem~\ref{thm.meta1}. The bicomodule $C^\l$ has a right adjoint ${}^\l C$ in $\Bicom \M$. The unit of adjunction is $\Delta_C$, and the counit is $l_0$. In the case of $\B = \Bicom \M$, the above properties of 2-cells $\theta$  correspond precisely to the bi-naturality. 
\end{proof}

A monad on an internal category $(\cC,\cA)$ is defined as a monad in the 2-category $\Int\M$ with underlying 0-cell $(\cC,\cA)$. That is, a monad is a triple $\t = (\t ,\mu, \eta)$, where $\t$ is an internal endofunctor on $(\cC,\cA)$, and $\mu : \t\t \Rightarrow \t$, $\eta: \id \Rightarrow \t$ are natural transformations satisfying the standard associativity and unitality conditions:
$$
\mu * t_1\mu = \mu * \mu t_0, \qquad \mu *t_1\eta = \mu *\eta t_0 = t . 
$$
Any adjunction $\l \dashv \r$ from $(\cC,\cA)$ to  $(\cD,\cB)$ with counit $\eps$ and unit $\eta$ defines a monad on $(\cC,\cA)$ by $(\r\l, r_1\eps l_0 , \eta)$. The aim of this section is to construct explicitly the Kleisli object (or the Kleisli internal category) for a monad $\t$ in $\Int\M$.

We first recall the definition of  Kleisli objects (in $\Int \M$) from \cite{KelStr:rev}; see also   \cite[Section~5]{MacSob:asp}. Start with  a monad $\t = (\t,\mu,\eta)$ on an internal category $(\cC,\cA)$ in $\M$. Then the $\t$-algebra 2-functor 
$$
\Alg \t : \Int \M \to \mathbf{CAT},
$$
is defined as follows. For every internal category $(\cD,\cB)$, $\Alg\t (\cD,\cB)$ is a category  consisting of pairs $(\y,\sigma)$, where $\y: (\cC,\cA) \to (\cD,\cB)$ is an internal functor and $\sigma: \y\t \Rightarrow \y$ is an internal natural transformation such that
$$
\sigma * y_1\eta = y, 
\qquad \sigma *y_1\mu = \sigma *\sigma t_0.
$$
A morphism $(\y,\sigma) \to (\y'\sigma')$ in  $\Alg\t (\cD,\cB)$ is an internal natural transformation $\alpha: \y\to \y'$ such that
$$
\alpha *\sigma = \sigma'*\alpha t_0.
$$
Given an internal functor $\f : (\cD,\cB) \to (\cD',\cB')$, $\Alg \t (\f) = f^\t$, where $f^\t (\y,\sigma) = (\f\y, f_1\sigma)$ and $f^\t(\alpha) = f_1\alpha$. Finally, given an internal natural transformation $\beta : \f \Rightarrow \g$, $\Alg\t (\beta)$ is a natural transformation $\beta^\t : \f^\t \Rightarrow \g^\t$ given on objects $(y,\sigma)$ in terms of the internal natural transformation $\beta y_0: \f\y \Rightarrow \g\y$. A {\em Kleisli object} for $\t$ is then defined as the representing object for $\Alg\t$.

\begin{theorem}\label{thm.kleisli}
Let  $\t = (\t ,\mu, \eta)$ be a monad on an internal category $(\cC,\cA)$ in $\M$. Denote by   $\cA_\t$ a $\cC$-bicomodule with underlying object
$
A_\t = C^\t \coten C A
$
and coactions $\Delta_C\coten C A$, $C^\t \coten C \varrho$. 
Define bicomodule maps
$$
u_\t : C\to A_\t, \qquad u_\t = (C\coten C \eta)\circ\Delta_C,
$$
$$
m_\t = ( C^\t \coten C A)\coten C( C^\t \coten C A) \cong  C^\t \coten C A^\t \coten C A \to A_\t,
$$
$$
m_\t = (C^\t \coten C m^2_A)\circ (C^\t\coten C \mu \coten C t_1\coten C A)\circ (\Delta_C\coten C A^\t\coten C A).
$$
Then $( \cC, \cA_\t, m_\t, u_\t)$ is an internal category in $\M$ and it is the Kleisli object for $\t$.
\end{theorem}
\begin{proof}
Let $\Phi: \Int\M \to \KL{\Bicom\M}$ be the 2-functor defined in the proof of Theorem~\ref{thm.meta1}. Since $\t$ is a monad on $(\cC,\cA)$, $\Phi(\t) = (\cC^\t, \phi_t)$ is a monad on $\Phi(\cC,\cA) = (\cC,\cA)$ in the bicategory $\KL{\Bicom\M}$, i.e.\ it is the (Kleisli) {\em wreath} in terminology of \cite{LacStr:for}. As explained in \cite[Section~3]{LacStr:for} the Kleisli object for $\t$ is the {\em wreath product} of the monad $(\cC,\cA)$ with $\Phi(\t)$. This is a monad in $\Bicom\M$ on $\cC$ with the endomorphism part $C^\t \coten C A = A_\t$. The multiplication is the composite
$$
\xymatrix{C^\t \coten C A^\t \coten C A\ar[rr]^-{C^\t \coten {} \phi_\t \coten{} A} && C^{\t \t }\coten C A\coten C A \ar[r]^-{C^{\t \t }\coten {} m_A} & C^{\t \t }\coten C A \ar[r]^-{\Phi(\mu)\coten C A} & C^\t \coten C A\coten C A \ar[r]^-{C^\t \coten {} m_A} & A_\t ,}
$$
where the standard isomorphisms are already taken into account. An easy calculation confirms that this composite is equal to $m_\t$. The unit in the wreath product $A_\t$ is given by $\Phi(\eta) = u_\t$. 
\end{proof}

One of the consequences of Theorem~\ref{thm.kleisli} is the existence of the {\em Kleisli adjunction} $\w\dashv \v$, $\xymatrix{(\cC,\cA)\ar@<0.5ex>[r]^\w & (\cC, \cA_\t). \ar@<0.5ex>[l]^\v}$ Explicitly, the functors $\v = (r_0, r_1)$ and  $\w = (l_0, l_1)$ are
$$
r_0  = t_0, \qquad  r_1: A_\t\to A, \quad r_1 = m_A \circ (\mu \coten C t_1),
$$
$$
l_0 = C, \qquad l_1: A\to A_\t, \quad l_1 = (C^\t\coten C m_A)\circ (C \coten C \eta\coten C A)\circ \lambda^2.
$$
The unit of the adjunction is $\eta$ (the unit of the monad $\t$), while the counit is
$
\eps = (C^\t \coten C t)\circ \Delta_C$. Note that $(\t,\mu,\eta) = (\r\l, r_1\eps l_0, \eta)$, as required. Note further that the bi-natural isomorphism $\theta$ corresponding to the Kleisli adjunction $\w\dashv \v$ by Theorem~\ref{thm.adjoint} is the identity. This is in perfect accord with the intuition coming from the standard ({\em external}) category theory in which the Kleisli adjunction isomorphism is the identity by the very definition of morphisms in the Kleisli category.

The Kleisli adjunction can be used to describe explicitly the  isomorphism of categories 
$$
\Theta: \Alg\t (\cD,\cB) \to \Int \M \left(\left(\cC, \cA_\t\right), \left(\cD,\cB\right)\right),
$$
that is natural in $(\cD,\cB)$ and whose existence is the contents of Theorem~\ref{thm.kleisli}. Here $\Int \M \left(\left(\cC, \cA_\t\right), \left(\cD,\cB\right)\right)$ denotes the category with objects internal functors $\left(\cC, \cA_\t\right) \to  \left(\cD,\cB\right)$ and morphisms internal natural transformations. On objects $(\y,\sigma)$ of $\Alg\t (\cD,\cB)$, $\Theta (\y,\sigma)$ is an internal functor with components
$$
\Theta (\y,\sigma)_0 = y_0, \qquad \Theta (\y,\sigma)_1 = m_B\circ (\sigma \coten D y_1) : C^\t \coten C A \to B.
$$
On morphisms $\Theta$ is the identity map. The inverse of $\Theta$ is given on objects as follows. Take an internal functor $\g: (\cC, \cA_\t)\to (\cD,\cB)$, and define
$$
\Theta^{-1}(\g) = (\y, \sigma),
\qquad 
\y = \g\w, \qquad \sigma = g_1\eps l_0 = g_1\eps .
$$

A {\em comonad} in a 2-category $\mathbf{K}$ is the same as a monad in a vertically dual 2-category $\mathbf{K}_*$, obtained from $\mathbf{K}$ by reversing the   2-cells. Thus, in the case of $\Int \M$, a comonad on an internal category $(\cC,\cA)$ is the same as a monad in $\Int \M_*$, i.e.\ a triple $(\g , \delta, \eps)$, where $\g$ is an internal endofunctor on $(\cC,\cA)$ and $\delta: \g \Rightarrow \g\g$, $\eps: \g\Rightarrow  \id$ are internal natural transformations such that
$$
g_1\delta *\delta = \delta g_0 *\delta, \qquad g_1\eps * \delta = \eps g_0*\delta = g.
$$
Since, by Theorem~\ref{thm.meta1}, there is a (locally full) embedding $\Int \M_* \to \KL{\Bicom \M^*}$,  internal comonads have {\em co-Kleisli} objects. As was the case for monads, these are given in terms of a (co)wreath product, and the explicit form of the embedding $\hPhi$ in the proof of Theorem~\ref{thm.meta1}, yields

\begin{theorem}\label{thm.main.comonad} Let  $\g = (\g ,\delta, \eps)$ be a comonad on an internal category $(\cC,\cA)$ in $\M$. Then $\g$ has a co-Kleisli object $(\cC,{}_\g\cA)$, where 
$$
{}_\g A = A \coten C {}^\g C,
$$
with the unit 
$
u_\g : C\to {}_\g A$, $u_\g = (\eps \coten C C)\circ\Delta_C$, 
and mutliplication
$$
m_\g = ( A \coten C {}^\g C)\coten C( A \coten C {}^\g C) \cong  A \coten C {}^\g A \coten C {}^\g C \to {}_\g A,
$$
$$
m_\g = (m^2_A \coten C {}^\g C)\circ (A \coten C g_1 \coten C \delta \coten C {}^\g C)\circ (A \coten C {}^\g A \coten C\Delta_C ).
$$
\end{theorem}

It has been observed already in \cite{EilMoo:adj} that a monad structure on a functor induces a comonad structure on its adjoint. This observation remains true in any 2-category, and the relationship between monad and comonad structures is that of being {\em mates under the adjunction} \cite{KelStr:rev}. It is also well known that the Kleisli category of a monad with a left adjoint is isomorphic to the co-Kleisli category of its comonad mate. The same property is enjoyed by monads on an internal category.
More explicitly, consider an adjunction $\l \dashv \r : (\cC,\cA)\to  (\cC,\cA)$ with unit $\iota$ and counit $\sigma$. If $(\r,\mu,\eta)$ is a monad, then $\l$ is a comonad with comultiplication $\delta$ being a mate of $\mu$,
$$
\delta = \sigma l_0^2 *l_1\mu l_0^2*l_1r_1\iota l_0 *l_1\iota,
$$
and the counit $\eps$ being a mate of $\eta$, $\eps = \sigma *l_1\eta$. By \cite[2.14~Theorem]{Lau:fro}, the Kleisli object $A_\r$ of $(\r,\mu,\eta)$ is isomorphic to the co-Kleisli object ${}_\l A$ of $(\l,\delta,\eps)$. In particular, the adjunction isomorphism
$
\theta:\!\!\xymatrix{A_\r = C^\r\coten C A \ar[r]^\simeq & A\coten C {}^\l C= {}_\l A, }
$
constructed in the proof of Theorem~\ref{thm.adjoint} is an isomorphism of internal categories.

\subsection{Kleisli objects in $\Intc \M$}
An internal cofunctor $\l : (\cC,\cA)\to (\cD,\cB)$ is said to be a {\em left adjoint} of a cofunctor $\r : (\cD,\cB)\to (\cC,\cA)$ in case $\l \dashv \r$ form an adjunction in $\Intc \M$. 
A monad in the 2-category $\Intc \M$ is called an {\em internal opmonad} (or simply an {\em opmonad}). Thanks to the existence of the locally full embedding $\Psi: \Intc\M \to \KL{\Bicom \M}$ described in (the proof of) Theorem~\ref{thm.meta2}, every opmonad has its Kleisli internal category, which is given in terms of a wreath product. Formally, one can formulate the following

\begin{theorem}\label{thm.kleisli.op}
Let  $\t = (\t ,\mu, \eta)$ be an opmonad on an internal category $(\cC,\cA)$ in $\M$. Then the  $\cC$-bicomodule ${}^\t\! \cA$ is a monoid in ${}^\cC \M^\cC$ with unit $\eta$ and multiplication
$$
m^\t = m^2_A\circ ( \mu \coten C t_1\coten C A)\circ (\lambda \coten C \lambda).
$$
The internal category $( \cC,{}^\t\!\cA, m^\t, \eta)$  is the Kleisli object for $\t$ (in $\Intc\M$).
\end{theorem}
\begin{proof}
This is simply the wreath product of $(\cC,\cA)$ with $\Psi(\t) = ({}^\t\cC, t_1)$, once the identification ${}^\t C \coten C A \cong {}^\t\! A$ is taken into account. Since $\Psi$ is the identity on natural cotransformations, the unit for ${}^\t\! \cA$ is the same as the unit of the opmonad, and the multiplication of the opmonad enters the wreath product unmodified. 
\end{proof}

\section{Examples} \label{sec.ex}
The aim of this section is to explore the meaning of internal adjunctions and internal (co)Kleisli categories in purely algebraic terms.  As the Kleisli objects considered here are wreath products in the opposite of the category of modules (over a commutative ring) the examples described in this section form a (dual) special case of the situation studied in \cite{ElK:ext}.
Throughout this section, $k$ is a commutative ring, the multiplication in $k$-algebras is denoted by juxtaposition and the unit is denoted by $1$. The unadorned tensor product is the usual tensor product of $k$-modules.

\subsection{Twisting of corings and algebras}
Let $A$ be a $k$-algebra. An {\em $A$-coring twisting datum} is a triple 
$$
\left( \xymatrix{ \cD \ar@<0.5ex>[r]^-{l} & \cC\ar@<0.5ex>[l]^-{r}}, \theta\right),
$$
in which $\cC = (C,\Delta_C, e_C)$ and $\cD = (D,\Delta_D, e_D)$ are $A$-corings, $l,r$ are maps of $A$-corings and $\theta : D^l \to {}^r\! C$ is an isomorphism of $\cD$-$\cC$-bicomodules. Recall that by an $A$-coring map we mean a morphism of comonoids in the monoidal category of $A$-bimodules. Thus $l$ and $r$ are $A$-bimodule maps compatible with comultiplications and counits. Recall further that $D^l$ means a $\cD$-$\cC$ bicomodule with the left coaction given by the comultiplication $\Delta_D$ and with the right coaction modified by $l$, $(D\ot_A l)\circ \Delta_D$. The meaning of ${}^r\! C$  is similar (but now left coaction is modified by $r$).

The readers can easily convince themselves that an $A$-coring twisting datum is the same as the internal adjunction $(A,l)\dashv (A,r)$ from $(A,\cC)$ to $(A,\cD)$ in $\Int {\mod_k^{op}}$. The bicomodule map property of $\theta$ is precisely the same as its bi-naturality. With this interpretation in mind, with no effort and calculations one proves
\begin{lemma}
In an $A$-coring twisting datum $
\left( \xymatrix{ \cD \ar@<0.5ex>[r]^-{l} & \cC\ar@<0.5ex>[l]^-{r}}, \theta\right)$, in which $e_D\circ \theta^{-1}$ and $e_C\circ\theta$ are units in convolution algebras $\lrhom A A {C} A$ and $\lrhom A A {D} A$ respectively, the $A$-corings $\cD$ and $\cC$ are mutually isomorphic.
\end{lemma}
\begin{proof}
In view of the proof of Theorem~\ref{thm.adjoint}, $e_D\circ \theta^{-1}$ is a unit of adjunction $(A,l)\dashv (A,r)$ and $e_C\circ\theta$ is the counit of $(A,l)\dashv (A,r)$. To require that they have convolution inverses is the same as to require that they are invertible 2-cells in $\Int {\mod_k^{op}}$. Thus the $A$-coring maps $l$ and $r$ must be isomorphisms. Explicitly, writing $u$ for $e_D\circ \theta^{-1}$,  $\bar{u}$ for its convolution inverse and $*$ for the convolution product, the inverse of $l$ is $l^{-1} = \bar{u}*r*u$.
\end{proof} 

The internal category interpretation of an  $A$-coring twisting datum $
\left( \xymatrix{ \cD \ar@<0.5ex>[r]^-{l} & \cC\ar@<0.5ex>[l]^-{r}}, \theta\right),
$ implies immediately the existence of   twisted $A$-coring structures on $\cC$ and $\cD$.
\begin{description}
\item[Twisting of $(C,\Delta_C,e_C)$]  The $A$-bimodule $C$ is an $A$-coring with comultiplication and counit
$$
\Delta_C^\theta = (\theta\circ r\ot_A C)\circ\Delta_C, \qquad e_C^\theta = e_D\circ \theta^{-1}.
$$
$\cC$ with this coring structure is denoted by $\cC_\theta$.
\item[Twisting of $(D,\Delta_D,e_D)$]  The $A$-bimodule $D$ is an $A$-coring with comultiplication and counit
$$
\Delta_D^\theta = (D\ot_A \theta^{-1}\circ l)\circ\Delta_D, \qquad e_D^\theta = e_C\circ \theta.
$$
$\cD$ with this coring structure is denoted by $\cD^\theta$.
\end{description}
$\cC_\theta$ is simply the Kleisli object of the monad associated to the adjunction $(A,l)\dashv (A,r)$, while $\cD^\theta$ is the co-Kleisli object of the corresponding comonad. The Kleisli adjunction yields a new $A$-coring twisting datum
 $$
\left( \xymatrix{ \cC_\theta \ar@<0.5ex>[r]^-{\bar{l}} & \cC\ar@<0.5ex>[l]^-{\bar{r}}, \bar{\theta}}\right),
$$
in which $\bar{r} = \theta\circ r$, $\bar{l} = (e_D\circ \theta^{-1}\ot_A C)\circ \Delta_C$ and $\bar{\theta}$ is the identity morphism on $C$. As explained in \cite[Section~1]{LacStr:for}, the process of associating Kleisli objects terminates after one step, hence this new twisting $A$-coring datum does not produce any new corings. More specifically, 
$$
\cC_{\bar{\theta}} = \cC_\theta \quad \mbox{and} \quad (\cC_\theta)^{\bar{\theta}} =\cC.
$$

Since a $k$-coalgebra is an example of a $k$-coring, the above constructions can be specified to $k$-coalgebras. In a similar way one can introduce an algebra twisting datum as a pair of $k$-algebra morphisms supplemented with a bimodule isomorphism (with bimodule structures suitably induced by algebra maps). To interpret such a datum as an internal adjunction one should (formally) assume that $k$ is a field, but algebraic arguments work for any commutative ring $k$. This is in fact an ``external'' construction since any $k$-algebra is the same as a $k$-linear category with a single object, and hence an algebra twisting datum is simply an adjunction (in the most classical sense) between $k$-linear categories with single objects.

\subsection{Sweedler's corings and Hopf-Galois extensions}
The aim of this section is to illustrate the construction of Kleisli objects on a very specific example and to make a connection with the notions appearing in Hopf-Galois theory. 

 To any  morphism of $k$-algebras $B\to A$ one associates the canonical Sweedler $A$-coring $A\ot_B A$ with coproduct and counit
$$
\Delta(a\ot_B a') = a\ot_B 1_A\ot_A 1_A\ot_B a', \qquad e(a\ot_B a') = aa'.
$$
This way $(A, A\ot_B A)$ becomes an internal category in $\mod_k^{op}$. To illustrate the construction of Kleisli objects we will study internal endofunctors $\t$ on $(A, A\ot_B A)$ which operate trivially on the object component of $(A, A\ot_B A)$, i.e.\ $\t = (A, t_1)$. We freely use the standard isomorphism, for all $A$-bimodules $M$,
$$
\lrhom AA{A\ot_B A} M \cong M^B := \{m\in M\; |\; \forall b\in B, \; bm=mb\}.
$$
With this identification at hand, a monad (which is an identity of the object of objects $A$) is a triple $t := \sum_i s_i\ot_B  t_i \in (A\ot_B A)^B$, $m,u\in A^B$ satisfying the following equations:
\begin{description}
\item[(a) $t$ is a functor] $\sum_i s_it_i =1$, $\sum_{i,j} s_i\ot_B t_is_j\ot_B t_j =\sum_i s_i\ot_B 1_A\ot_B t_i$;
\item[(b) naturality of multiplication] $tm = mt^2$;
\item[(c) naturality of unit] $tu = u\ot_B 1_A$;
\item[(d) the associative law] $m^2 = \sum_i ms_imt_i$;
\item[(e) the unit law] $mu = \sum_i ms_iut_i =1_A$.
\end{description}
In (b), $t^2 = \sum_i s_itt_i =\sum_{i,j} s_is_j \ot_B t_jt_i$, i.e.\ it is the square of $t$ with respect to the natural multiplication in $(A\ot_B A)^B$. For example, any unit $u\in A^B$ defines a monad on $(A, A\ot_B A)$ as a triple $(u\ot_B u^{-1}, u^{-1}, u)$. That is, for all $a,a'\in A$,  $t_0(a) = a$, $t_1(a\ot_B a') = au\ot_B u^{-1}a'$, $\mu (a\ot_B a') = au^{-1}a'$, $\eta(a\ot_B a') = aua'$. 

The Kleisli coring $(A\ot_B A)_\t$ associated to the datum $(t,m,u)$ equals $A\ot_B A$ as an $A$-bimodule, but has modified comultiplication and counit,
$$
\Delta_t (a\ot_B a') = amt\ot_B a', \qquad e_t(a\ot_B a') = aua'.
$$
\begin{lemma}\label{lem.cor}
The element $mt$ is a group-like element in $(A\ot_B A)_\t$.
\end{lemma}
\begin{proof}
The unit law (e) implies immediately that $e_t(mt) =1$. Furthermore,
\begin{eqnarray*}
\Delta_t(mt) &=& \sum_i ms_imt\ot_B t_i = \sum_{i,j} ms_i m t_i s_jt\ot_B t_j\\
&=& \sum_i m^2 s_it\ot_B t_i = \sum_i m^2 s_i tt_i t\\
&=&\sum_im^2 t^2s_i\ot_B t_i = mt\ot_A mt,
\end{eqnarray*}
where the second and fourth equalities follow by (a), the third one by (d) and the final one by (b).
\end{proof} 

The Kleisli adjunction comes out as
$$
l_1: A\ot_B A \to (A\ot_B A)_t, \qquad a\ot_B a' \to amta',
$$
$$
r_1: (A\ot_B A)_t \to A\ot_B A, \qquad a\ot_B a' \to  au\ot_B a'.
$$
Lemma~\ref{lem.cor} makes is clear that $l_1$ is an $A$-coring map as required. That $r_1$ is an $A$-coring map follows by (c) and (e). Of course, both statements follow by general arguments described in Section~\ref{sec.Kleisli}. The unit of the Kleisli adjunction is given by $e_t$ and the counit by $e$. 

This description of Kleisli objects for monads on the Sweedler canonical coring can be transferred to Hopf-Galois extensions. Let $H$ be a Hopf algebra, and $A$ be a right $H$-comodule algebra with coaction $\varrho^A : A\to A\ot H$. Let $B = A^{coH} = \{a\in A\; |\; \varrho^A(a) = a\ot 1_H\}$ be the coinvariant subalgebra of $A$. Then the inclusion $B\subseteq A$ is called a {\em Hopf-Galois extension} provided the canonical map
$$
\can: A\ot_B A\to A\ot H, \qquad a'\ot_B a \mapsto a'\varrho^A(a),
$$
is an isomorphism; see e.g.\ \cite[Chapter~8]{Mon:Hop}. The restricted inverse of the canonical map
$$
\tau: H\to A\ot_BA, \qquad h\mapsto \can^{-1}(1_A\ot h),
$$
is called the {\em translation map}. The translation map has several nice properties (see \cite[3.4~Remark]{Sch:rep}), which in particular imply that, for any group-like element $x\in H$ (i.e.\ an element such that $\Delta_H(x) = x\ot x$ and $e_H(x) =1$), $\tau(x) \in (A\ot_B A)^B$ satisfies conditions (a). Thus any group-like element $x\in H$ induces an endofunctor on $A\ot_B A\cong A\ot H$. Using further properties of the translation map, one finds that an element $m\in A^B$ satisfies condition (b) if and only if 
$$
\varrho^A (m) = m\ot x.
$$
Similarly, an element $u\in A^B$ satisfies condition (c) if and only if
$$
\varrho^A (u) = u\ot x^{-1},
$$
(note that this equality makes sense, since every group-like element of a Hopf algebra is necessarily a unit). The Galois condition induces an action of $H$ on $A^B$, known as the {\em Miyashita-Ulbrich action} \cite{Ulb:Gal}, \cite{DoiTak:Hop}. For any $h\in H$, write $\tau(h) = \sum_i h \suc 1_i \ot_B h\suc 2 _i$. Then, for any $a\in A^B$, the Miyashita-Ulbrich action is given by
$$
a\ra h := \sum \sum_i h \suc 1_i a h\suc 2 _i.
$$
In the Hopf-Galois case, the associative and unital laws (d) and (e) come out as
$$
m^2 = m(m\ra h), \qquad mu = m(u\ra x) =1.
$$

\section*{Acknowledgements} 
Some of the results of this paper were presented at CT2009 in Cape Town. The authors are grateful to many participants of CT2009, in particular to Gabriella B\"ohm and Ross Street, 
for comments and suggestions. 
T.\ Brzezi\'nski would like to acknowledge the International Travel Grant from  the Royal Society (U.K.). The research of A.\ Vazquez Marquez is supported by the CONACYT  grant no.\  208351/303031

\end{document}